\documentclass[12pt,leqno]{amsart}
\usepackage[latin9]{inputenc}
\usepackage{verbatim}
\usepackage{amsthm}
\usepackage{amstext}
\usepackage{amssymb}
\usepackage[dvipdfmx]{graphicx}
\usepackage{xcolor}
\usepackage{soul}
\makeatletter
\numberwithin{equation}{section}
\numberwithin{figure}{section}

\usepackage{amsthm}\usepackage{amscd}\usepackage{bm}

\usepackage[dvipdfmx,pdfstartview=FitH,pdfborderstyle={/S/B/W 1}]{hyperref}

\theoremstyle{plain}
\newtheorem{main theorem}{Main Theorem}
\newtheorem{theorem}{Theorem}[section]
\newtheorem{lemma}[theorem]{Lemma}

\newtheorem{corollary}[theorem]{Corollary}
\newtheorem{proposition}[theorem]{Proposition}
\newtheorem{claim}[theorem]{Claim}

\newtheorem{subclaim}[theorem]{Subclaim}
\theoremstyle{definition}
\newtheorem{definition}[theorem]{Definition}
\newtheorem{remark}[theorem]{Remark}
\newtheorem{example}[theorem]{Example}

\newtheorem{problem}[theorem]{Problem}

\newcommand{\mdim}{\mathrm{mdim}}
\newcommand{\diam}{\mathrm{diam}}
\newcommand{\widim}{\mathrm{Widim}}

\makeatother

\begin{document}

\title[Expansive multiparameter actions and mean dimension]{Expansive multiparameter actions and mean dimension}

\author{Tom Meyerovitch, Masaki Tsukamoto}



\date{\today}

\thanks{}

\maketitle

\begin{abstract}
Ma\~{n}\'{e} (1979) proved that if a compact metric space admits an expansive homeomorphism then
it is finite dimensional.
We generalize this theorem to multiparameter actions.
The generalization involves mean dimension theory, which counts ``averaged dimension''
of a dynamical system.
We prove that if $T:\mathbb{Z}^k\times X\to X$ is expansive and if 
$R:\mathbb{Z}^{k-1}\times X\to X$ commutes with $T$ then
$R$ has finite mean dimension.
When $k=1$, this statement reduces to Ma\~{n}\'{e}'s theorem.
We also study several related issues, especially the connection with entropy theory.
\end{abstract}

\section{Introduction}  \label{section: introduction}

\subsection{Main results} \label{subsection: main results}

Let $(X,d)$ be a compact metric space.
A homeomorphism $T:X\to X$ is said to be \textbf{expansive} if there exists $c>0$ such that any
distinct two points $x$ and $y$ in $X$ satisfy
\[ \sup_{n\in \mathbb{Z}} d(T^n x, T^n y) >c. \]
Hyperbolic dynamics provides many examples of expansive maps \cite[Chapter 3]{Bowen75}:
\textit{A diffeomorphism is expansive on hyperbolic sets}.
Motivated by the work of Bowen \cite{Bowen70} on hyperbolic minimal sets,
Ma\~{n}\'{e} \cite{Mane} investigated the topological dimension of a compact metric space admitting an expansive homeomorphism:

\begin{theorem}[Ma\~{n}\'{e}, 1979]  \label{theorem: Mane1}
   Let $T:X\to X$ be an expansive homeomorphism. Then $X$ is finite dimensional.
\end{theorem}

Therefore infinite dimensional spaces (e.g. the infinite dimensional cube $[0,1]^{\mathbb{N}}$)
cannot admit expansive homeomorphisms.
This is a rather surprising result because the definition of expansiveness
seems to have nothing to do with dimension theory.
Fathi \cite[Corollaries 5.4 and 5.5]{Fathi} revisited this phenomena from the viewpoint of entropy theory and he proved:

\begin{theorem}[Fathi, 1989]  \label{theorem: Fathi}
 Let $T:X\to X$ be an expansive homeomorphism.
 If the topological entropy $h_{\mathrm{top}}(T)$ is zero, then $X$ is zero dimensional.
\end{theorem}

The main purpose of this paper is to generalize the above theorems of Ma\~{n}\'{e} and
Fathi to \textit{multiparameter actions} (i.e. \textbf{$\mathbb{Z}^k$-actions}).
A continuous action $T:\mathbb{Z}^k\times X\to X$ on a compact metric space $(X,d)$ is
said to be \textbf{expansive} if there exists $c>0$ such that any two distinct points $x$ and $y$ in $X$
satisfy $\sup_{u \in \mathbb{Z}^k} d(T^u x, T^u y) >c$.
At first sight, it looks \textit{nonsense} to study $\mathbb{Z}^k$-versions of Theorems \ref{theorem: Mane1} and \ref{theorem: Fathi}
because we can easily find examples which seemingly deny this direction:

\begin{example}  \label{example: easy counter-example}
  \begin{enumerate}
   \item Let $\mathbb{T}^2 = \mathbb{R}^2/\mathbb{Z}^2$ be the standard two dimensional torus and
   $h:\mathbb{T}^2\to \mathbb{T}^2$ a hyperbolic toral automorphism,  e.g.
    $h= \begin{pmatrix} 2 & 1 \\ 1& 1\end{pmatrix}$.
    $h$ is an expansive homeomorphism.
    Consider the two-sided infinite product $X:=\left(\mathbb{T}^2\right)^{\mathbb{Z}}$ (with the product topology) and let
    $\sigma: X\to X$ be the shift.
    Define $h_\mathbb{Z}:X\to X$ by $h_\mathbb{Z}(x_n)_{n\in \mathbb{Z}} = \left(h(x_n)\right)_{n\in \mathbb{Z}}$.
   Then $\sigma$ and $h_\mathbb{Z}$ generate an expansive $\mathbb{Z}^2$-action on $X$
    although $X$ is infinite dimensional
    (Shi--Zhou \cite[Proposition 4.2]{Shi--Zhou}).
    See also Example \ref{example: expansive algebraic action} below for a different kind of expansive actions
    on infinite dimensional spaces.
    \item Let $\mathrm{id}:\mathbb{T}^2\to \mathbb{T}^2$ be the identity map.
    Then $\mathrm{id}$ and the hyperbolic toral automorphism $h:\mathbb{T}^2\to \mathbb{T}^2$
    generate an expansive $\mathbb{Z}^2$-action on $\mathbb{T}^2$ with zero topological entropy
    although $\mathbb{T}^2$ is positive dimensional.
    With a bit more effort, we can construct an expansive $\mathbb{Z}^2$-action of zero topological entropy on 
    an infinite dimensional space. 
    (Let $D=0$ in Example \ref{example of zero mean dimension} below.)
   \end{enumerate}
\end{example}

The above examples show that
we cannot naively generalize Theorems \ref{theorem: Mane1} and \ref{theorem: Fathi} to $\mathbb{Z}^k$-actions.
We have to change our viewpoint.

It turns out that \textbf{mean dimension theory} provides a reasonable framework for the problem.
This is a topological invariant of dynamical systems introduced by Gromov \cite{Gromov},
which counts the number of variables \textit{per iterate} to describe a point in a dynamical system.
We denote by $\mdim(X,T)$ the mean dimension of a continuous action $T:\mathbb{Z}^k\times X\to X$.
We will review its definition in \S \ref{section: mean dimension}.
It is known that (if $k\geq 1$) mean dimension is zero for all finite dimensional systems and
finite topological entropy systems.
The $\mathbb{Z}^k$-shift on $\left([0,1]^D\right)^{\mathbb{Z}^k}$ has mean dimension $D$ and
the $\mathbb{Z}^k$-shift on $\left([0,1]^{\mathbb{N}}\right)^{\mathbb{Z}^k}$ has infinite mean dimension.

Mean dimension has applications to topological dynamics
\cite{Lindenstrauss--Weiss, Lindenstrauss, Gutman 1, Gutman--Lindenstrauss--Tsukamoto,
Gutman--Tsukamoto minimal, Gutman--Qiao--Tsukamoto}.
As an illustration, we review one result \cite[Main Theorem 1]{Gutman--Qiao--Tsukamoto}
which originated in the work of Lindenstrauss \cite[Theorem 5.1]{Lindenstrauss}:
\textit{If a  free minimal $\mathbb{Z}^k$-action $(X,T)$ satisfies
$\mdim(X,T) < D/2$ then we can equivariantly embed it in the $\mathbb{Z}^k$-shift on $\left([0,1]^D\right)^{\mathbb{Z}^k}$.}
The value $D/2$ here is optimal.
This shows in particular that a  free minimal $\mathbb{Z}^k$-action  $(X,T)$ has finite mean dimension if and only if one   can equivariantly embed it in the $\mathbb{Z}^k$-shift on $\left([0,1]^D\right)^{\mathbb{Z}^k}$, for some finite $D$.

The following is our first main result.

\begin{theorem} \label{Theorem: Mean dimension of co-rank 1 commuting action}
Let $T:\mathbb{Z}^k\times X\to X$ be an expansive action on a compact metric space $X$,
and let $R:\mathbb{Z}^{k-1} \times X \to X$ be a continuous action that commutes with $T$. 
Namely, $T^{v} \circ R^{u} = R^{u} \circ T^{v}$ for all $v \in \mathbb{Z}^k$ and $u \in \mathbb{Z}^{k-1}$.
Then  
\[  \mdim\left(X,R\right) < \infty. \]
\end{theorem}

For a subgroup $A\subset \mathbb{Z}^k$ we denote by $T|_A: A\times X \to X$ the restriction of $T$ to $A$.
Letting $R=T|_A$ with $\mathrm{rank}\, A = k-1$ we have the following special case:

\begin{corollary}  \label{corollary of  first main theorem}
Let $T:\mathbb{Z}^k\times X\to X$ be an expansive action on a compact metric space $X$.
Then for any subgroup $A\subset \mathbb{Z}^k$ with $\mathrm{rank}\, A = k-1$
\[  \mdim\left(X,T|_A\right) < \infty. \]
Namely, the restriction of $T$ to any rank $(k-1)$ subgroup has finite mean dimension.
\end{corollary}

\begin{remark}
   Here are several remarks on the theorem:
  \begin{itemize}
     \item When $k=1$, an action $R:\mathbb{Z}^0 \times X \to X$ is the trivial action,  and the mean dimension
     $\mdim\left(X, R\right)$ is equal to the topological dimension $\dim X$.
     Thus the statement of Theorem \ref{Theorem: Mean dimension of co-rank 1 commuting action} and also of Corollary \ref{corollary of  first main theorem} reduce to the original theorem of
     Ma\~{n}\'{e} (Theorem \ref{theorem: Mane1}) in this case.
     \item In Example \ref{example: easy counter-example} (1), the mean dimension of $\sigma$ is two and
     the mean dimension of $h_\mathbb{Z}$ is zero.
     \item Expansive actions always have finite topological entropy.
     So the mean dimension of an expansive action $T:\mathbb{Z}^k\times X\to X$ itself is zero
     since finite topological entropy systems are zero mean dimensional.
     This is trivial and uninteresting.
     The point of Corollary \ref{corollary of  first main theorem} is that it provides a nontrivial information of
     actions of \textit{infinite index subgroups}.
     Our viewpoint is summarized by the following correspondence:
     \[  \text{Expansiveness of $\mathbb{Z}^k$-actions} \longleftrightarrow
          \text{mean dimension of $\mathbb{Z}^{k-1}$-actions}. \]
  \end{itemize}
\end{remark}

\begin{example}  \label{example: expansive algebraic action}
Let $\mathbb{T} = \mathbb{R}/\mathbb{Z}$ be the circle and consider the infinite
product $\mathbb{T}^{\mathbb{Z}^2}$ index by $\mathbb{Z}^2$.
Let $\sigma$ be the $\mathbb{Z}^2$-shift on it.
We define a $\sigma$-invariant closed subset $X$ of $\mathbb{T}^{\mathbb{Z}^2}$ by
\[  X= \left\{(x_{mn})_{(m,n)\in \mathbb{Z}^2}\in \mathbb{T}^{\mathbb{Z}^2} \middle|\,
     \forall (m,n)\in \mathbb{Z}^2: \,  3x_{m n} + x_{m+1, n} + x_{m, n+1} = 0 \right\}. \]
Then $(X,\sigma)$ is expansive.
(This fact can be directly and easily checked.
See Schmidt \cite{Schmidt} for a more general and systematic study of this kind of examples.)
We can check that for any rank one subgroup $A\subset \mathbb{Z}^2$ the mean dimension
$\mdim\left(X,\sigma|_A\right)$ is positive and finite.
\end{example}

We will provide two proofs of Theorem \ref{Theorem: Mean dimension of co-rank 1 commuting action}.
The first proof (given in \S \ref{section: first proof of first main theorem})
has the advantage that it is elementary and self-contained.
It uses only the definitions of mean dimension.
The second proof (given in \S \ref{section: second proof of first main theorem and the proof of second main theorem})
requires more machineries, in particular Lindenstrauss--Weiss' theory of
\textit{metric mean dimension} \cite{Lindenstrauss--Weiss}.
The advantage of the second approach is that it also provides the following generalization of
Fathi's theorem (Theorem \ref{theorem: Fathi}).
This is our second main result:

\begin{theorem} \label{second main theorem}
 Let $T:\mathbb{Z}^k\times X\to X$ be an expansive action on a compact metric space $X$ and
 let $R:\mathbb{Z}^{k-1} \times X \to X$ be a continuous action that commutes with $T$.
 If the topological entropy of $T$ is zero, then 
\[ \mdim\left(X, R \right) = 0. \]
\end{theorem}

\begin{corollary} \label{corollary of second main theorem}
 Let $T:\mathbb{Z}^k\times X\to X$ be an expansive action on a compact metric space $X$.
 If the topological entropy of $T$ is zero, then for any subgroup $A\subset \mathbb{Z}^k$ with $\mathrm{rank}\, A = k-1$
 \[ \mdim\left(X, T|_A\right) = 0. \]
\end{corollary}

We would like to note that the expansiveness is an essential assumption in the statement of Theorem \ref{second main theorem}
(i.e. the zero entropy of $T$ alone does not imply $\mdim(X, R)=0$).
For example, the identity map $\mathrm{id}$ and the $\mathbb{Z}^k$-shift $\sigma$
on $[0,1]^{\mathbb{Z}^k}$ generate a (non-expansive) zero entropy $\mathbb{Z}^{k+1}$-action on $[0,1]^{\mathbb{Z}^k}$
although the mean dimension $\mdim\left([0,1]^{\mathbb{Z}^k},\sigma\right)$ is positive (equal to one).

\begin{example}   \label{example of zero mean dimension}
Let $h:\mathbb{T}^r \to \mathbb{T}^r$ be
 a hyperbolic toral automorphism as in Example \ref{example: easy counter-example} (1), where $r = 2$.
Let $\Lambda$ be a subset of $\mathbb{Z}$.
The \textbf{upper Banach density of} $\Lambda$ is given by
\[ D :=  \lim_{N\to \infty} \frac{\sup_{n\in \mathbb{Z}} |\Lambda \cap [n, n+N)|}{N} \in [0,1]. \]
This limit exists because $\sup_{n\in \mathbb{Z}} |\Lambda \cap [n, n+N)|$ is a subadditive function in $N$.
Let $A\subset \mathbb{T}^2$ be a non-empty closed $h$-invariant subset such that the topological entropy
$h_{\mathrm{top}}(A,h)$ of the restriction of $h$ on $A$ is zero.
For example we can choose $A= \{\text{the fixed point of $h$}\}$.
From Fathi's theorem (Theorem \ref{theorem: Fathi}), $A$ is zero dimensional.
Let $Y_0$ be the set of points $x=(x_n)_{n\in \mathbb{Z}}$ in $X = \left(\mathbb{T}^r\right)^\mathbb{Z}$
satisfying
\[  \forall n\in \mathbb{Z}\setminus \Lambda: \quad x_n \in A. \]
This is $h_\mathbb{Z}$-invariant but not $\sigma$-invariant.
We define $Y\subset X$ by
\[  Y = \overline{\bigcup_{n\in \mathbb{Z}} \sigma^n\left(Y_0\right)}. \]
$Y$ is invariant under both $h_\mathbb{Z}$ and $\sigma$.
So they generate an expansive $\mathbb{Z}^2$-action on $Y$.
The topological entropy of this $\mathbb{Z}^2$-action is given by
\[ h_{\mathrm{top}}\left(Y, \sigma, h_{\mathbb{Z}} \right) = h_{\mathrm{top}}(h) \, D. \]
Here $h_{\mathrm{top}}(h)$ is the topological entropy of $h: \mathbb{T}^r \to \mathbb{T}^r$.
The above formula can be verified by a direct computation,
or by using the following well-known formula  (see \cite[Lemma $3.1$]{Pavlov}):
\[
h_{\mathrm{top}}\left(Y, \sigma, h_{\mathbb{Z}} \right) = \lim_{n \to \infty} \frac{1}{N} h_{\mathrm{top}}\left(\pi_N(Y),h_N\right),\]
where $\pi_N:Y \to (\mathbb{T}^r)^{\{1,\ldots,N\}}$ is the obvious projection map and $h_N:\pi_N(Y) \to \pi_N(Y)$ is defined by applying $h$ on each coordinate.

On the other  hand the mean dimension of $\sigma$ on $Y$ is given by
\[ \mdim\left(Y, \sigma \right) = r D. \]
This formula relies on the fact that $\widim_\varepsilon(\mathbb{T}^m, \ell^{\infty}) = m$ 
for sufficiently small $\varepsilon$ independent of $m$, 
where $\ell^{\infty}$ is the metric on $\mathbb{T}^m$ that comes from  $\|\cdot\|_\infty$ norm on $\mathbb{R}^m$ 
(see Lemma \ref{lemma: widim of M^n}  below). 

We see that in this case $\mdim\left(Y,\sigma \right)$ becomes zero exactly
when $h_{\mathrm{top}}\left(Y, \sigma, h_{\mathbb{Z}}\right)$ is zero.
This behavior is (of course) compatible with the statement of Theorem \ref{second main theorem}.
We also would like to remark that
if $A$ is an infinite set then
the topological entropy of $\sigma$ on $Y$ is always infinite regardless of
the value of $D$.
Thus it is not the topological entropy $h_{\mathrm{top}}(Y,\sigma)$ but the mean dimension $\mdim(Y,\sigma)$
that reflects the circumstances properly.
\end{example}

\subsection{Jointly expansive automorphisms of $\mathbb{Z}^k$-actions} \label{subsection: automorphisms}

Here we discuss the materials in \S \ref{subsection: main results} from the viewpoint of \textit{automorphisms} of dynamical systems.
An advantage of this approach is that we can also apply it to general amenable group actions.
(See \S \ref{subsection: noncommutative versions} below.)

\begin{definition}  \label{definition: automorphism}
  Let $T:\mathbb{Z}^k\times X\to X$ be a continuous action (not necessarily expansive) on a compact metric space $X$.

   \begin{enumerate}
     \item A homeomorphism $f:X\to X$ is called an \textbf{automorphism} of $(X,T)$ if it commutes with
     the $T$-action: $T^u \circ f = f\circ T^u$ for all $u\in \mathbb{Z}^k$.

     \item An automorphism $f$ of $(X,T)$ is said to be \textbf{jointly expansive} if $f$ and $T$ generate
     an expansive $\mathbb{Z}^{k+1}$-action.
  \end{enumerate}
\end{definition}

\begin{example}  \label{example: jointly expansive automorphism}
The following are examples of existence/non-existence of jointly expansive automorphisms:

\begin{enumerate}
\item
Example \ref{example: easy counter-example} (1) shows that the shift $\sigma$ on $\left(\mathbb{T}^2\right)^{\mathbb{Z}}$
admits a jointly expansive automorphism $h_{\mathbb{Z}}$.

\item
The $\mathbb{Z}^k$-shift $\sigma$ on $[0,1]^{\mathbb{Z}^k}$ does not admit a jointly expansive automorphism:
If $f:[0,1]^{\mathbb{Z}^k}  \to [0,1]^{\mathbb{Z}^k}$ is a jointly expansive automorphism, then
it yields an expansive \textit{homeomorphism} on a fixed point set $\mathrm{Fix}(\sigma)$ of $\sigma$.
But $\mathrm{Fix}(\sigma)$ is homeomorphic to the unit interval $[0,1]$, which does not admit an expansive homeomorphism
(cf. \cite[Proposition 1.1.6]{Katok--Hasselblatt}).
\end{enumerate}
\end{example}

Example \ref{example: jointly expansive automorphism} (2)
shows that the set of periodic points are (sometimes) obstructions to the existence of
jointly expansive automorphisms.
But if a system is free (i.e. it has no periodic points), then we cannot use this obstruction.
Mean dimension provides another obstruction:

\begin{corollary}  \label{corollary of main theorem}
If a $\mathbb{Z}^k$-action $(X,T)$ admits a jointly expansive automorphism, then
$\mdim(X,T)$ is finite.
\end{corollary}

This is an immediate corollary of Theorem \ref{Theorem: Mean dimension of co-rank 1 commuting action}.
By using the method of Lindenstrauss--Weiss \cite[Proposition 3.5]{Lindenstrauss--Weiss},
we can easily construct plenty of examples of free (and, moreover, minimal) $\mathbb{Z}^k$-actions
of infinite mean dimension.
Such systems do not admit jointly expansive automorphisms although they have no periodic points.

Of course, in general, neither periodic points nor mean dimension provide
a sufficient criterion for the existence of a jointly expansive automorphism.
For example, an irrational rotation on the circle does not admit
a jointly expansive automorphism\footnote{A circle homeomorphism commuting with an irrational rotation
must be a rotation. It is proved in \cite[Theorem 3.1]{Shi--Zhou} that the circle does not admit an
expansive $\mathbb{Z}^k$-action for any $k\geq 1$.}
although it is free and zero mean dimensional.

\subsection{On expansive and minimal $\mathbb{Z}^k$-actions and mean dimension of lower rank subgroups}

As we briefly noted in the beginning, the original motivation of Ma\~{n}\'{e} came from
Bowen's work \cite{Bowen70}.
Bowen \cite{Bowen70} proved that hyperbolic minimal sets of a diffeomorphism are always zero dimensional.
Ma\~{n}\'{e} \cite{Mane} generalized this to a more abstract setting
(see also Artigue \cite{Artigue} for a recent new proof):

\begin{theorem}[Ma\~{n}\'{e}, 1979]  \label{theorem: Mane2}
   If $f:X\to X$ is an expansive and minimal homeomorphism on a compact metric space $X$, then $X$ is zero dimensional.
   Here $f$ is said to be minimal if the orbit $\{f^n x\}_{n\in \mathbb{Z}}$ is dense in $X$ for every $x\in X$.
\end{theorem}

Contrary to Theorems \ref{theorem: Mane1} and \ref{theorem: Fathi}, we do not currently have an appropriate
multiparameter version of Theorem \ref{theorem: Mane2}.
The following results preclude some seemingly plausible generalizations:

\begin{proposition}  \label{proposition: counter-example of multidimensional Mane2}
  There exists a positive mean dimensional $\mathbb{Z}$-action $(X,T)$ admitting a jointly expansive and
  minimal automorphism $f:X\to X$.
\end{proposition}

\begin{proposition} \label{proposition: second counter-example of multidimensional Mane2}
There exists a minimal and expansive $\mathbb{Z}^2$-action  with the property that for every line $L \subset \mathbb{R}^2$ the directional mean dimension of the action with respect to $L$ is
positive.
\end{proposition}

\textit{Directional mean dimension} is a mean dimension analogue of \textit{directional entropy}
(Milnor \cite{Milnor} and Boyle--Lind \cite{Boyle--Lind}) that was suggested recently by Lind. It counts the averaged dimension of $(X,T)$ along the $L$ direction.
We will define it in \S \ref{subsection: directional mean dimension}.
Propositions \ref{proposition: counter-example of multidimensional Mane2} and \ref{proposition: second counter-example of multidimensional Mane2} both show that a for $k >1$,  a rank $(k-1)$-subaction of a minimal and expansive $\mathbb{Z}^k$-action  need not have 
zero mean dimension, in contrast to the case $k=1$. Furthermore, proposition   \ref{proposition: counter-example of multidimensional Mane2} shows that this can happen even when  an single element of $\mathbb{Z}^2$  acts minimally, and proposition 
\ref{proposition: second counter-example of multidimensional Mane2} shows in particular that a minimal and expansive $\mathbb{Z}^2$-action  can have positive mean dimension for \emph{every} element of  $\mathbb{Z}^2$.

We will prove proposition \ref{proposition: counter-example of multidimensional Mane2} in \S
\ref{section: proof of Proposition counter example} and proposition \ref{proposition: second counter-example of multidimensional Mane2} in \S \ref{section: directional mean dimension and multidimensional Mane2}.
The question remains:

\begin{problem}
   Is there a reasonable generalization of Theorem \ref{theorem: Mane2} to $\mathbb{Z}^k$-actions?
\end{problem}

\subsection{Noncommutative versions} \label{subsection: noncommutative versions}

We can consider generalizations
of \S \ref{subsection: main results} and \S \ref{subsection: automorphisms}
 to noncommutative group actions.
Let $(X,d)$ be a compact metric space.

\begin{description}
   \item[Polynomial growth groups]
   Let $G$ and $H$ be finitely generated groups of polynomial growth.
   We denote by $\deg(G)$ and $\deg(H)$ the degrees of the polynomial growth of $G$ and $H$ respectively
   (e.g. $\deg(\mathbb{Z}^k) = k$).

   \begin{theorem} \label{theorem: polynomial growth}
    Let $T:G\times X\to X$ and $R:H\times X\to X$ be continuous actions which commute with each other.
   Suppose $T$ is expansive, namely there exists $c>0$ such that any distinct $x,y\in X$ satisfy 
   $\sup_{g\in G} d(T^g x, T^g y) > c$.
     \begin{enumerate} 
       \item Suppose $\deg(G) = \deg(H)+1$. Then:
           \begin{enumerate}
               \item   The mean dimension $\mdim(X,R)$ is finite.
               \item If the topological entropy of $T$ is zero then $\mdim(X,R)=0$.
           \end{enumerate} 
      \item  Suppose $\deg(G)= \deg(H)$. Then:
         \begin{enumerate}
             \item  The topological entropy of $R$ is finite.    
             \item If the topological entropy of $T$ is zero then the topological entropy of $R$ is also zero.
          \end{enumerate}   
      \item Suppose $\deg(G) < \deg(H)$. Then the topological entropy of $R$ is zero.
     \end{enumerate}   
   \end{theorem}
   
   The case (1) is the most nontrivial case with respect to the viewpoint of mean dimension theory.
   The mean dimension of $R$ is zero in the cases 
   (2) and (3) because finite topological entropy systems are zero mean dimensional.
   Indeed, as we saw at the end of Example \ref{example of zero mean dimension}, 
   finiteness of topological entropy is a strictly stronger condition than zero mean dimensionality.
   
   \begin{remark}
   The case (b) of (2) and the case (3) above were already proved by Shereshevsky \cite{Shereshevsky}.
   \end{remark}

   \item[Amenable groups]
    Amenable groups may have exponential growth. So we cannot apply the framework of Theorem \ref{theorem: polynomial growth}
    to general amenable groups.
    However the formulation in \S \ref{subsection: automorphisms} using automorphisms can be naturally generalized to
    amenable group actions.
     Let $G$ be a finitely generated amenable group and
    $T:G \times X\to X$ a continuous action.
     A homeomorphism $f:X\to X$ is called an automorphism of $(X,T)$ if it commutes with
       the $T$-action.
        An automorphism $f$ is said to be jointly expansive if there exists $c>0$ satisfying
        $\sup_{n\in \mathbb{Z}, g\in G} d(f^n T^g x, f^n T^g y) > c$ for any two distinct $x,y\in X$.

        \begin{theorem}  \label{theorem: amenable group case}
            Suppose a $G$-action $(X,T)$ admits a jointly expansive automorphism $f$. Then:
          \begin{enumerate}
            \item The mean dimension $\mdim(X,T)$ is finite.
             \item If the topological entropy of the $G \times \mathbb{Z}$-action generated by $T$ and $f$ is zero, then
          the mean dimension $\mdim(X,T)$ is zero.
           \end{enumerate}
        \end{theorem}
\end{description}

The proofs of Theorems \ref{theorem: polynomial growth} and \ref{theorem: amenable group case} 
are straightforward generalizations\footnote{The cases (2) and (3) in Theorem \ref{theorem: polynomial growth} 
did not appear in \S \ref{subsection: main results} (at least formally).
But indeed they naturally follow if we apply the argument of
\S \ref{section: second proof of first main theorem and the proof of second main theorem}     
(or, more simply, Lemmas \ref{lemma: long time boundedness} and \ref{lem: Coding commuting action}
in \S \ref{section: first proof of first main theorem})
to the settings of polynomial growth group actions.} of the proofs of
Theorems \ref{Theorem: Mean dimension of co-rank 1 commuting action} and \ref{second main theorem}.
But we have not so far found any interesting phenomena specific to the noncommutative case.
So the main body of the paper concentrates on the case of $\mathbb{Z}^k$-actions
and we omit the detailed explanations of the noncommutative case.
We believe that experienced readers will not find any difficulty to
extend the arguments of \S \ref{section: first proof of first main theorem}
and \S \ref{section: second proof of first main theorem and the proof of second main theorem}
to noncommutative group actions.

\vspace{0.2cm}

\textbf{Acknowledgment.}
The authors started the research of this paper when they attended the conference
``Mean dimension and sofic entropy meet dynamical systems, geometric analysis and
information theory'' at the Banff International Research Station (BIRS) for Mathematical Innovation and
Discovery.
The authors would like to thank BISR for their wonderful environment and hospitality.

\section{Mean dimension} \label{section: mean dimension}

Here we review basics of mean dimension.
Readers can find (much) more information in \cite{Gromov, Lindenstrauss--Weiss, Lindenstrauss}.

Let $(X,d)$ be a compact metric space.
Let $\mathcal{U} = \{U_i\}_{i\in I}$ be an open cover of $X$.
We define $\mathrm{mesh}(\mathcal{U},d)$ as the supremum of $\diam(U_i)$ over $U_i\in \mathcal{U}$.
We define the \textbf{order} $\mathrm{ord}(\mathcal{U})$ as the maximum integer $n\geq 0$ such that
there exist pairwise distinct $i_0,i_1,\dots,i_n\in I$ satisfying
$U_{i_0}\cap U_{i_1}\cap \dots \cap U_{i_n} \neq \emptyset$.
An open cover $\mathcal{V} = \{V_j\}_{j\in J}$ of $X$ is called a \textbf{refinement} of $\mathcal{U}$ if
for every $j\in J$ there exists $i\in I$ satisfying $V_j\subset U_i$.
We define the \textbf{degree} $\mathcal{D}(\mathcal{U})$ as the minimum of $\mathrm{ord}(\mathcal{V})$
over all refinements $\mathcal{V}$ of $\mathcal{U}$.
It is known \cite[Definition 1.6.7]{Engelking} that the topological dimension $\dim X$ is give by
the supremum of $\mathcal{D}(\mathcal{U})$ over all open covers $\mathcal{U}$ of $X$.

For two open covers $\mathcal{U} = \{U_i\}_{i\in I}$ and $\mathcal{V} = \{V_j\}_{j\in J}$ of $X$
we define a new open cover $\mathcal{U}\vee \mathcal{V}$ by
\[ \mathcal{U}\vee \mathcal{V} = \{U_i\cap V_j|\, i\in I, j\in J\}. \]
We can check that \cite[Corollary 2.5]{Lindenstrauss--Weiss}
\begin{equation}  \label{eq: subadditivity of degree}
    \mathcal{D}(\mathcal{U}\vee \mathcal{V}) \leq \mathcal{D}(\mathcal{U}) + \mathcal{D}(\mathcal{V}).
\end{equation}

Suppose $\mathbb{Z}^k$ continuously acts on $X$ by $T:\mathbb{Z}^k\times X\to X$.
We define the \textbf{mean dimension} $\mdim(X,T)$ by
\begin{equation} \label{eq: definition of mean dimension}
    \mdim(X,T) = \sup_{\text{$\mathcal{U}$: open cover of $X$}}
   \left(\lim_{N\to \infty} \frac{\mathcal{D}\left(\bigvee_{u\in [-N,N]^k\cap \mathbb{Z}^k} T^{-u}\mathcal{U}\right)}{(2N+1)^k}\right).
\end{equation}
This limit exists because of the subadditivity (\ref{eq: subadditivity of degree}).
The mean dimension is a topological invariant of $(X,T)$.

The above formulation (\ref{eq: definition of mean dimension}) was introduced by
\cite{Lindenstrauss--Weiss}.
Another formulation (closer to the original definition of \cite{Gromov})
will be also useful later (\S \ref{section: first proof of first main theorem}):
For $\varepsilon >0$ we define the \textbf{$\varepsilon$-width dimension}
$\widim_\varepsilon (X,d)$ as the minimum of $\mathrm{ord}(\mathcal{U})$ over
all open covers $\mathcal{U}$ of $X$ satisfying $\mathrm{mesh}(\mathcal{U},d) \leq \varepsilon$.
Given an action $T:\mathbb{Z}^k \times X \to X$ and a subset $\Omega\subset \mathbb{R}^k$, 
the distance $d^T_\Omega$ on $X$ is defined by
\[ d^{T}_\Omega(x,y) = \sup_{u\in \Omega\cap \mathbb{Z}^k} d(T^u x, T^u y). \]
When there is no ambiguity about the action $T$ we will write $d_\Omega = d^T_\Omega$.

Then $\mdim(X,T)$ is given by
\begin{equation}  \label{eq: another definition of mean dimension}
   \mdim(X,T) = \lim_{\varepsilon\to 0}
   \left(\lim_{N\to \infty} \frac{\widim_\varepsilon\left(X,d^{T}_{[-N,N]^k}\right)}{(2N+1)^k}\right).
\end{equation}
The equivalence of (\ref{eq: definition of mean dimension}) and (\ref{eq: another definition of mean dimension})
easily follows from the consideration on the \textit{Lebesgue number}.
The above definitions are all we need in \S \ref{section: first proof of first main theorem}.
So readers may skip the rest of this section and directly go to \S \ref{section: first proof of first main theorem}.

Next we introduce \textit{metric mean dimension} \cite{Lindenstrauss--Weiss}.
This will be used in \S \ref{section: second proof of first main theorem and the proof of second main theorem}.
Let $\varepsilon >0$ and $(X,d)$ a compact metric space.
We define $\#(X,d,\varepsilon)$ as the minimum cardinality $|\mathcal{U}|$ of open covers $\mathcal{U}$ of $X$
satisfying $\mathrm{mesh}(\mathcal{U}, d) < \varepsilon$.
Let $T:\mathbb{Z}^k\times X\to X$ be a continuous action.
We define the \textbf{entropy $S(X,T,d,\varepsilon)$ at the scale $\varepsilon>0$} by
\begin{equation}  \label{eq: elementary bound on S}
  S(X,T,d,\varepsilon) := \lim_{N\to \infty}
     \frac{\log \#\left(X, d^{T}_{[-N,N]^k}, \varepsilon\right)}{(2N+1)^k}
     = \inf_{N\geq 1} \frac{\log \#\left(X,d^T_{[-N,N]^k},\varepsilon\right)}{(2N+1)^k}.
\end{equation}     
The second equality follows from a \textit{standard deviation argument}.
The topological entropy $h_{\mathrm{top}}(T)$ is given by
\begin{equation}  \label{eq: topological entropy}
   h_{\mathrm{top}}(T) = \lim_{\varepsilon \to 0} S(X,T,d,\varepsilon).
\end{equation}
We define the \textbf{upper/lower metric mean dimensions} $\overline{\mdim}(X,T,d)$ and
$\underline{\mdim}(X,T,d)$ by
\begin{equation*}
    \overline{\mdim}(X,T,d) = \limsup_{\varepsilon \to 0} \frac{S(X,T,d,\varepsilon)}{\log(1/\varepsilon)}, \quad
    \underline{\mdim}(X,T,d) = \liminf_{\varepsilon \to 0} \frac{S(X,T,d,\varepsilon)}{\log(1/\varepsilon)}.
\end{equation*}
The following is a fundamental theorem \cite[Theorem 4.2]{Lindenstrauss--Weiss}.

\begin{theorem}[Lindenstrauss--Weiss, 2000]  \label{theorem: Lindenstrauss--Weiss}
   \[   \mdim(X,T) \leq  \underline{\mdim}(X,T,d) \leq    \overline{\mdim}(X,T,d). \]
\end{theorem}

\section{First proof of Theorem \ref{Theorem: Mean dimension of co-rank 1 commuting action}} \label{section: first proof of first main theorem}

Here we prove Theorem \ref{Theorem: Mean dimension of co-rank 1 commuting action}
by adapting Ma\~{n}\'{e}'s method \cite[pp. 318-319]{Mane} to the settings of mean dimension theory.
Throughout this section we assume that $(X,d)$ is a compact metric space with an expansive
action $T:\mathbb{Z}^k\times X\to X$,
and that $R:\mathbb{Z}^{k-1} \times X \to X$ is another action that commutes with $T$.
We choose $c>0$ such that any two distinct points $x,y\in X$ satisfy
\begin{equation}  \label{eq: expansivity constant}
   \sup_{u\in \mathbb{Z}^k} d(T^u x, T^u y) > 2c.
\end{equation}

\begin{lemma}  \label{lemma: boundary and interior behaviors}
   There exists $\delta>0$ such that if $N \geq 1$ and $x,y\in X$ satisfy
   \[ c\leq d^T_{[-N,N]^k}(x,y) \leq 2c, \]
   then $d^T_{\partial [-N,N]^k}(x,y) > \delta$.
   Here $\partial [-N,N]^k$ is the boundary of $[-N,N]^k$, i.e. it is given by
   \[ \bigcup_{i=1}^k \left\{x\in [-N,N]^k|\, x_i \in \{-N,N\} \right\}. \]
\end{lemma}

\begin{proof}
Suppose the statement is false: There exist
$N_n\geq 1$ and $x_n, y_n\in X$ $(n\geq 1)$ satisfying
\[ c\leq d^T_{[-N_n,N_n]^k}(x_n,y_n) \leq 2c, \quad
    \lim_{n\to \infty} d^T_{\partial [-N_n,N_n]^k}(x_n,y_n) =0. \]
Then $N_n\to \infty$ as $n\to \infty$ and there exists $a_n\in [-N_n,N_n]^k$ satisfying
$d(T^{a_n}x_n, T^{a_n}y_n) \geq c$.
It follows from $\lim_{n\to \infty} d^T_{\partial [-N_n,N_n]^k}(x_n,y_n) =0$ that
the distance between $a_n$ and $\partial [-N_n,N_n]^k$ goes to infinity as $n\to \infty$.
We can assume that $T^{a_n} x_n\to x$ and $T^{a_n} y_n\to y$ by choosing subsequences (if necessary).
Then $d(x,y)\geq c$ and $\sup_{u\in \mathbb{Z}^k} d(T^u x, T^u y) \leq 2c$.
This contradicts (\ref{eq: expansivity constant}).
\end{proof}

\begin{lemma}  \label{lemma: long time boundedness}
For any $\varepsilon >0$ there exists $m=m(\varepsilon)>0$ such that if $x,y\in X$ satisfy
\[  d^T_{[-m,m]^k}(x,y) \leq 2c, \]
then $d(x,y)  < \varepsilon$.
\end{lemma}

\begin{proof}
Suppose the statement is false:
There exist $\varepsilon >0$ and $x_n,y_n\in X$ $(n\geq 1)$ satisfying
$d^T_{[-n,n]^k}(x_n,y_n) \leq 2c$ and
$d(x_n,y_n) \geq \varepsilon$.
Choose subsequences $\{x_{n'}\}$ and $\{y_{n'}\}$ converging to some $x$ and $y$ respectively.
Then $\sup_{u\in \mathbb{Z}^k} d(T^u x, T^u y) \leq 2c$ and $d(x,y) \geq \varepsilon$,
which contradicts (\ref{eq: expansivity constant}).
\end{proof}

\begin{proposition}  \label{main proposition}
\[  \limsup_{N\to \infty} \frac{\widim_{2c}\left(X, d^T_{[-N,N]^k}\right)}{N^{k-1}} < \infty. \]
\end{proposition}

\begin{proof}
Let $\delta>0$ be the constant introduced in Lemma \ref{lemma: boundary and interior behaviors}.
Choose an open cover $\mathcal{U}= \{U_1,\dots, U_L\}$ of $X$ with
$\mathrm{mesh}(\mathcal{U},d) < \delta$.
For $N\geq 1$ we consider the open cover
\begin{equation} \label{eq: boundary condition}
    \bigvee_{u\in \partial [-N,N]^k}T^{-u}\mathcal{U}.
\end{equation}
It follows from (\ref{eq: subadditivity of degree}) in \S \ref{section: mean dimension}
that
\[ \mathcal{D}\left( \bigvee_{u\in \partial [-N,N]^k}T^{-u}\mathcal{U}\right)  \leq
    \left|\mathbb{Z}^k\cap \partial [-N,N]^k\right|\cdot \mathcal{D}(\mathcal{U})  \leq
     2^k (2N+1)^{k-1} L. \]
Thus there exists a refinement $\mathcal{V}_N$ of (\ref{eq: boundary condition}) satisfying
$\mathrm{ord}(\mathcal{V}_N) \leq 2^k (2N+1)^{k-1}L$.

Take $V\in \mathcal{V}_N$.
We define an equivalence relation on $V$ as follows: For $x,y\in V$ we write $x\sim_V y$ if there exists a
finite sequence $x_0,x_1,\dots,x_n$ in $V$ satisfying
\begin{equation}  \label{eq: equivalence relation}
     x_0 = x, \quad x_n = y, \quad \forall 0\leq i <n: \> d^T_{[-N,N]^k}(x_i, x_{i+1}) < c.
\end{equation}
Let $V= V_1\cup \dots \cup V_{a(V)}$ be the decomposition into the equivalence classes.
Set $\mathcal{W}_N = \{V_i|\, V\in \mathcal{V}_N, 1\leq i \leq a(V)\}$.
That is, $\mathcal{W}_N$ is obtained from  $\mathcal{V}_N$ 
by breaking its elements into  ``$c$-approximately connected components'' with respect to the metric $d^{ T}_{[-N,N]^k}$.
This is an open cover of $X$ with
\begin{equation}  \label{eq: order of W_N}
     \mathrm{ord}(\mathcal{W}_N) = \mathrm{ord}(\mathcal{V}_N) \leq 2^k (2N+1)^{k-1} L.
\end{equation}

\begin{claim}  \label{claim: mesh of W_N}
    \[  \mathrm{mesh}(\mathcal{W}_N, d^T_{[-N,N]^k}) \leq 2c. \]
\end{claim}

\begin{proof}
 Suppose the statement is false: There exist $V\in \mathcal{V}_N$ and $x,y\in V$ satisfying
 $x\sim_V y$ and $d_{[-N,N]^k}(x,y) > 2c$.
 It follows from the definition of $\sim_V$ that we can find $x_0, \dots, x_n$ in $V$ satisfying (\ref{eq: equivalence relation}).
 Then some $x_i$ must satisfy $c\leq d^T_{[-N,N]^k}(x, x_i) \leq 2c$.
  Since $\mathcal{V}_N$ is a refinement of (\ref{eq: boundary condition}), it also satisfies
  \[ d^T_{\partial [-N,N]^k}(x,x_i) \leq  \mathrm{mesh}  (\mathcal{U}, d) < \delta. \]
  This contradicts Lemma \ref{lemma: boundary and interior behaviors}.
\end{proof}

From Claim \ref{claim: mesh of W_N} and (\ref{eq: order of W_N})
\[ \widim_{2c}\left(X, d_{[-N,N]^k}\right) \leq  \mathrm{ord}\left(\mathcal{W}_{N}\right)
   \leq 2^k (2N+1)^{k-1}L. \]
Thus we get
\[  \limsup_{N\to \infty} \frac{\widim_{2c} \left(X, d^{T}_{[-N,N]^k}\right)}{N^{k-1}}  \leq   2^{2k-1} L. \]
\end{proof}

The following lemma enables us to control the $R$-action by the information of the $T$-action.
This is contained in Shereshevsky \cite[Lemma 2.2]{Shereshevsky}.

\begin{lemma}\label{lem: Coding commuting action}
There exists $K >0$ such that for every $N >0$  the following holds:
If $x,y \in X$ satisfy 
$$d^T_{[-KN,KN]^k}(x,y) \leq 2c,$$
then
$$d^R_{[-N,N]^{k-1}}(x,y) \leq 2c.$$
\end{lemma}

\begin{proof}
There exists $\varepsilon >0$ such that 
$$d(R^u x,R^u y)\leq 2c \mbox{ for all } u \in \{-1,0,1\}^{k-1}$$
whenever $d(x,y) < \varepsilon$.
By Lemma \ref{lemma: long time boundedness} there exists $K >0$  so that $d^T_{[-K,K]^k}(x,y) \leq 2c$ implies $d(x,y) < \varepsilon$. 
For $A \times B \subset \mathbb{R}^{k-1} \times \mathbb{R}^{k}$ write
$$d^{(R,T)}_{A \times B}(x,y) = \max_{u \in A\cap \mathbb{Z}^{k-1}, v \in B \cap \mathbb{Z}^k}d(R^u(T^v(x)),R^u(T^v(y)).$$
Suppose $d^T_{[-KN,KN]^k}(x,y) \leq 2c$.
Then it follows by induction on $0 \le j \le N$ that 
$$d^{(R,T)}_{[-j,j]^{k-1} \times [-K(N-j),K(N-j)]^k}(x,y) \leq 2c.$$
The claim of the lemma now follows by setting $j=N$.
\end{proof}

\begin{proof}[Proof of Theorem \ref{Theorem: Mean dimension of co-rank 1 commuting action}]
The mean dimension $\mdim(X,R)$ is given by
\[ \lim_{\varepsilon \to 0} \left(\lim_{N\to \infty}
   \frac{\widim_\varepsilon\left(X, d^R_{[-N,N]^{k-1}}\right)}{(2N+1)^{k-1}}\right). \]
By Lemma \ref{lemma: long time boundedness} for every $\epsilon >0$ there exists $m=m(\varepsilon)>0$ so that 
$$ d^{(R,T)}_{[-N,N]^{k-1}\times [-m,m]^k} (x,y) \leq 2c \> \Longrightarrow \> d^R_{[-N,N]^{k-1}}(x,y) < \varepsilon. $$
By Lemma \ref{lem: Coding commuting action} 
\[ d^T_{[-KN-m,KN+m]^k}(x,y) \leq 2c \Longrightarrow  d^{(R,T)}_{[-N,N]^{k-1}\times [-m,m]^k} (x,y) \leq 2c.\]
Hence
\[   \widim_\varepsilon\left(X, d^R_{[-N,N]^{k-1}}\right) \leq 
     \widim_{2c} \left(X, d^T_{[-KN-m,KN+m]^k}\right). \]
Noting $m = m(\varepsilon)$ is independent of $N$
\[   \lim_{N\to \infty} \frac{\widim_\varepsilon\left(X, d^R_{[-N,N]^{k-1}}\right)}{(2N+1)^{k-1}}
       \leq   (K/2)^{k-1} \limsup_{N\to \infty} \frac{\widim_{2c}\left(X, d^T_{[-N,N]^k}\right)}{N^{k-1}}. \]
By Proposition \ref{main proposition} the right-hand side is finite and independent of $\varepsilon$.
Thus $\mdim(X,R)$ is finite
\end{proof}

\begin{remark}
A similar argument shows that
there exists $C < \infty$ such that every rank $(k-1)$ subgroup $A\subset \mathbb{Z}^k$ satisfies
\[ \mdim(X,T|_A) \leq C \lim_{N\to \infty} \frac{N^{k-1}}{\left|A\cap [-N,N]^k\right|}. \]
\end{remark}

\section{Second proof of Theorem \ref{Theorem: Mean dimension of co-rank 1 commuting action} 
and the proof of Theorem \ref{second main theorem}}
\label{section: second proof of first main theorem and the proof of second main theorem}

Here we prove Theorems \ref{Theorem: Mean dimension of co-rank 1 commuting action} and \ref{second main theorem}
by adapting Fathi's method \cite[Section 5]{Fathi} to our settings.

\subsection{Frink's metrization theorem} \label{subsection: Frink's metrization theorem}

Here we review a classical theorem of Frink \cite{Frink}.

\begin{theorem}[Frink, 1937]  \label{theorem: Frink}
Let $X$ be a set and let $\rho$ a nonnegative function on $X\times X$ satisfying
\begin{enumerate}
   \item $\rho(x,y) = \rho(y,x)$ for all $x,y\in X$.
   \item $\rho(x,y)=0$ if and only if $x=y$.
   \item $\rho(x,z) \leq 2\max(\rho(x,y),\rho(y,z))$ for all $x,y,z\in X$.
\end{enumerate}
Then there exists a distance function $D$ on $X$ satisfying
\[  \frac{\rho(x,y)}{4} \leq D(x,y) \leq \rho(x,y), \quad (\forall x,y\in X). \]
Here ``distance function'' means that it satisfies the above (1), (2) and the triangle inequality.
\end{theorem}

\begin{proof}
We explain Frink's nice proof \cite[pp. 134-135]{Frink} for readers' convenience.

\begin{claim}  \label{claim: Frink}
For $x_0,x_1,\dots, x_n\in X$ with $n\geq 2$
\[  \rho(x_0,x_n) \leq 2\rho(x_0,x_1) + 4 \sum_{i=1}^{n-2} \rho(x_i,x_{i+1}) + 2\rho(x_{n-1},x_n). \]
When $n=2$, the right-hand side is just $2\rho(x_0,x_1) + 2\rho(x_1,x_2)$.
\end{claim}

Assuming this claim for the moment, we prove the theorem.
Let $x,y\in X$.
We define $D(x,y)$ as the infimum of $\sum_{i=0}^{n-1}\rho(x_i,x_{i+1})$ over all
$x_0,x_1,\dots, x_n\in X$ with $x_0=x$ and $x_n=y$.
It immediately follows that $D$ satisfies $D(x,y)=D(y,x)$, $D(x,y) \leq \rho(x,y)$ and the triangle inequality.
The inequality $\rho(x,y)/4\leq D(x,y)$ follows from Claim \ref{claim: Frink}.
Then $D$ satisfies (2) and becomes a distance function.

Now we start the proof of Claim \ref{claim: Frink}.
The proof is the induction on $n$.
The statement is true for $n=2$ by the condition (3).
Let $N\geq 3$ and suppose the statement is true for $n\leq N-1$.
Let $x_0,\dots,x_N\in X$. We can assume $x_0\neq x_N$.
We define $m\in [1,N]$ as the minimum integer satisfying $\rho(x_0,x_m)>\rho(x_m,x_N)$.
If $m=1$ then the statement follows from (3) because
\[ \rho(x_0,x_N) \leq 2\max\left(\rho(x_0,x_1),\rho(x_1,x_N)\right)  =2\rho(x_0,x_1). \]
If $m=N$ then $\rho(x_0,x_{N-1}) \leq \rho(x_{N-1},x_N)$ and hence the statement follows from
\[  \rho(x_0,x_N) \leq 2\max\left(\rho(x_0,x_{N-1}),\rho(x_{N-1},x_N)\right) = 2\rho(x_{N-1},x_N). \]
So we assume $2\leq m\leq N-1$. This implies
\[  \rho(x_0,x_{m-1}) \leq \rho(x_{m-1},x_N), \quad
     \rho(x_0,x_m) > \rho(x_m, x_N). \]
Hence by (3)
\[    \rho(x_0,x_N) \leq 2\max\left(\rho(x_0,x_m), \rho(x_m,x_N)\right) = 2\rho(x_0,x_m), \]
\[    \rho(x_0,x_N) \leq 2\max\left(\rho(x_0,x_{m-1}), \rho(x_{m-1},x_N)\right) = 2\rho(x_{m-1},x_N).  \]
By adding these two inequalities, we get
\[ \rho(x_0,x_N) \leq \rho(x_0,x_m) + \rho(x_{m-1},x_N). \]
By the induction hypothesis,
\[ \rho(x_0,x_m) \leq 2\rho(x_0,x_1) + 4\sum_{i=1}^{m-2} \rho(x_i,x_{i+1}) + 2\rho(x_{m-1},x_m), \]
\[  \rho(x_{m-1}, x_N) \leq 2\rho(x_{m-1},x_m) + 4\sum_{i=m}^{N-2}\rho(x_i,x_{i+1}) + 2\rho(x_{N-1},x_N). \]
By adding these two inequalities, we get the statement of the claim.
\end{proof}

\subsection{Proofs of Theorems \ref{Theorem: Mean dimension of co-rank 1 commuting action} and \ref{second main theorem}}
\label{subsection: proof of first and second main theorems}

Throughout this subsection we assume
that $(X,d)$ is a compact metric space with an
expansive action $T:\mathbb{Z}^k\times X\to X$, and that $R:\mathbb{Z}^{k-1} \times X \to X$ commutes with $T$.
We choose $c>0$ such that any two distinct points $x,y\in X$ satisfy
\[ \sup_{u \in \mathbb{Z}^k} d(T^u x, T^u y) > c. \]
Fix an integer $l>0$ such that if $x,y\in X$ satisfy $d(x,y) \geq c/2$ then there exists $u\in \mathbb{Z}^k$
with $|u| \leq l$ satisfying $d(T^u x, T^u y) \geq c$.
Fix $\alpha>1$ with $\alpha^l < 2$.

Let $x,y\in X$. We define
\[ \mathbf{n}(x,y) = \min \left\{n\geq 0|\, \exists u\in \mathbb{Z}^k: \text{ $|u|\leq n$ and $d(T^u x, T^u y) \geq c$} \right\}. \]
If $x=y$ then we set $\mathbf{n}(x,y) = \infty$.
We set $\rho(x,y) = \alpha^{-\mathbf{n}(x,y)}$.

\begin{lemma}  \label{lemma: properties of rho}
The function $\rho$ satisfies:
  \begin{enumerate}
    \item $\rho(x,y) = \rho(y,x)$.
    \item $\rho(x,y) = 0$ if and only if $x=y$.
    \item $\rho(x,z) \leq 2\max(\rho(x,y), \rho(y,z))$.
    \item If $x_n\to x$ and $y_n\to y$ in $X$ as $n\to \infty$ then
    \[ \limsup_{n\to \infty} \rho(x_n, y_n) \leq \rho(x,y). \]
    \item $\rho$ is compatible with the topology of $X$. Namely, the balls (with respect to $\rho$)
    \[ B_r(x,\rho) = \{y\in X|\, \rho(x,y) < r\} \quad (x\in X, r>0) \]
    form an open base of the topology of $X$.
  \end{enumerate}
\end{lemma}

\begin{proof}
(1) and (2) are obvious.

(3) We can assume $x\neq z$ and let $m= \mathbf{n}(x,z)$.
There exists $|u|\leq m$ with $d(T^u x, T^u z) \geq c$.
Then either $d(T^u x, T^u y)\geq c/2$ or $d(T^u y, T^u z) \geq c/2$.
Suppose the former. (The latter is the same.)
There exists $|v|\leq l$ with $d(T^{u+v}x, T^{u+v}y) \geq c$.
Then $\mathbf{n}(x,y) \leq m+l$ and hence
\[ \rho(x,y) = \alpha^{-\mathbf{n}(x,y)} \geq \alpha^{-m} \alpha^{-l} > \frac{\rho(x,z)}{2}, \]
where we used $\alpha^l < 2$.

(4) It is straightforward to check $\liminf_{n\to \infty} \mathbf{n}(x_n,y_n) \geq \mathbf{n}(x,y)$.

(5) It follows from the property (4) above that the balls $B_r(x,\rho)$ are open (with respect to $d$).
Expansiveness implies that for any $x\in X$ and $R>0$ there exists $r>0$ satisfying
$B_r(x,\rho) \subset B_R(x,d)$.
Then the statement can be easily checked.
\end{proof}

By the properties (1), (2), (3) of Lemma \ref{lemma: properties of rho} and
 Frink's metrization theorem (Theorem \ref{theorem: Frink}) we can
find a distance function $D$ on $X$ satisfying
\[ \frac{\rho(x,y)}{4}    \leq D(x,y) \leq  \rho(x,y). \]
By the property (5) of Lemma \ref{lemma: properties of rho}, the distance $D$ is compatible with the topology of $X$.

\begin{lemma} \label{lemma: property of D}
 If $n\geq 1$ and $x,y\in X$ satisfy $\max_{|u|< n} D(T^u x, T^u y) < 1/(4\alpha)$ then
 $D(x,y) < \alpha^{-n}$.
\end{lemma}

\begin{proof}
It follows $\max_{|u|<n} \rho(T^u x, T^u y) < 1/\alpha$ and hence
$\min_{|u|<n} \mathbf{n}(T^u x, T^u y) >1$.
This implies $\mathbf{n}(x,y) > n$ and $D(x,y) \leq \rho(x,y) < \alpha^{-n}$.
\end{proof}

Let $x\neq y$ be two points in $X$.
There exists $v \in \mathbb{Z}^k$ with $d(T^v x, T^v y) \geq c$ and hence $\rho(T^v x, T^v y) =1$.
Thus 
\[  \sup_{u \in \mathbb{Z}^k} D(T^u x, T^u y) \geq \frac{1}{4} > \frac{1}{4\alpha}. \]
By the same argument as in Lemma \ref{lem: Coding commuting action}, we can prove that there exists $K\geq 1$
such that for any $N\geq 1$
\[  D^T_{[-KN,KN]^k} (x,y) < \frac{1}{4\alpha} \Longrightarrow 
    D^R_{[-N,N]^{k-1}}(x,y) < \frac{1}{4\alpha}. \]
Then for any $n, N\geq 1$
\begin{equation*}
    \begin{split}
      D^T_{[-KN-n,KN+n]^k}(x,y) < \frac{1}{4\alpha}  &\Longrightarrow D^{(R,T)}_{[-N,N]^{k-1}\times [-n,n]^k}(x,y) < \frac{1}{4\alpha} \\
     &\Longrightarrow   D^R_{[-N,N]^{k-1}}(x,y) < \alpha^{-n} \quad 
     (\text{by Lemma \ref{lemma: property of D}}).
    \end{split}
\end{equation*}    
Thus 
\begin{equation}  \label{eq: bound on R by T}
  \#\left(X, D^R_{[-N,N]^{k-1}},\alpha^{-n}\right) \leq 
  \#\left(X, D^T_{[-KN-n,KN+n]^k}, 1/(4\alpha)\right).
\end{equation}

\begin{proof}[Proof of Theorems \ref{Theorem: Mean dimension of co-rank 1 commuting action} and \ref{second main theorem}]

We will prove 
\begin{equation} \label{eq: main statement}
    \overline{\mdim}(X,R,D) \leq 2(K+1)^k \frac{h_{\mathrm{top}}(T)}{\log\alpha}.
\end{equation}
This shows Theorems \ref{Theorem: Mean dimension of co-rank 1 commuting action} and \ref{second main theorem}
because $h_{\mathrm{top}}(T) < \infty$ for every expansive action and
 $\mdim(X,R) \leq \overline{\mdim}(X,R,D)$ by Theorem \ref{theorem: Lindenstrauss--Weiss}.

Let $\delta>0$ be arbitrary.
It follows from the definition of the topological entropy that there exists $M>0$ such that
for any $N\geq M$
\[  \frac{1}{(2N+1)^k} \log\#\left(X, D^{T}_{[-(K+1)N,(K+1)N]^k}, 1/(4\alpha)\right)
    < (K+1)^k (h_{\mathrm{top}}(T) + \delta). \]
Let $0<\varepsilon < \alpha^{-M}$.
Choose $N \geq M$ with $\alpha^{-N} < \varepsilon \leq \alpha^{-N+1}$.
By using (\ref{eq: bound on R by T}) with $n=N$
\begin{equation*}
   \begin{split}
     \#\left(X, D^R_{[-N,N]^{k-1}},\varepsilon\right) &\leq  \#\left(X, D^R_{[-N,N]^{k-1}},\alpha^{-N}\right)  \\
     &\leq \#\left(X, D^T_{[-KN-N,KN+N]^k}, 1/(4\alpha)\right).  
   \end{split}
\end{equation*}   
Thus 
\[  \log \#\left(X, D^R_{[-N,N]^{k-1}},\varepsilon\right)  < (2N+1)^k (K+1)^k (h_{\mathrm{top}}(T) + \delta).   \] 
By the second equality of (\ref{eq: elementary bound on S}) 
\[ S(X,R,D,\varepsilon) < (2N+1) (K+1)^k  (h_{\mathrm{top}}(T) + \delta).   \] 
From $\varepsilon \leq \alpha^{-N+1}$,
\[  N-1 \leq \frac{\log(1/\varepsilon)}{\log \alpha}. \]
Thus $S(X,R,D,\varepsilon)$ is bounded by
\[   \left(\frac{2\log(1/\varepsilon)}{\log\alpha} + 3\right) (K+1)^k \left(h_{\mathrm{top}}(T) + \delta\right). \]
So we get
\[ \overline{\mdim}(X,R,D) = \limsup_{\varepsilon \to 0} \frac{S(X,R,D,\varepsilon)}{\log(1/\varepsilon)}
   \leq    2 (K+1)^k \frac{h_{\mathrm{top}}(T) + \delta}{\log \alpha}. \]
Since $\delta>0$ is arbitrary, this proves (\ref{eq: main statement}).
\end{proof}

\subsection{Remark on entropy and metric mean dimension}  \label{subsection: remark on entropy and metric mean dimension}

The idea of the previous subsection is roughly summarized by the following correspondence:
\[  \text{Topological entropy of $\mathbb{Z}^k$-actions}   \longleftrightarrow
     \text{Metric mean dimension of $\mathbb{Z}^{k-1}$-actions}. \]
We will give one remark on this correspondence.
Let $(X,T)$ be a $\mathbb{Z}^k$-action (not necessarily expansive) 
and let  $d$ be a  metric that generates the topology of $X$ such that
 $\underline{\mdim}(X,T,d) < \infty$.
Let $f:X\to X$ be an \textbf{endomorphism}\footnote{Namely $f$ is a continuous  map (not necessarily invertible)
from $X$ to $X$ satisfying $f\circ T^u = T^u \circ f$ for all $u\in \mathbb{Z}^k$.}
of $(X,T)$ that is Lipschitz with respect to the metric $d$. The \textit{local Lipschitz constant of $f$} is given by
\begin{equation}  \label{eq: Local Lipschitz constant}
   L :=  \lim_{\varepsilon \to 0} \sup_{0<d(x,y)<\varepsilon} \frac{d(f(x), f(y))}{d(x,y)} < \infty.
\end{equation}
We denote $\log^+ L = \max\left(0, \log L\right)$.
\begin{proposition}  \label{proposition: entropy and metric mean dimension}
Under the above circumstances,
\[   h_{\mathrm{top}}(T,f) \leq \log^+ L\cdot  \underline{\mdim}(X,T,d), \]
where the left-hand side is the topological entropy of the $\mathbb{Z}^k \times \mathbb{Z}_{\geq 0}$-action generated by $T$ and $f$.
\end{proposition}

\begin{proof}
For $\Omega \subset \mathbb{R}^k \times \mathbb{R}_{\geq 0}$ we define a distance $d_\Omega$ on $X$ by
\[   d_\Omega (x,y) = \sup_{(u,n)\in \Omega\cap (\mathbb{Z}^k\times \mathbb{Z}_{\geq 0})}
     d\left(T^u\circ f^n(x), T^u\circ f^n(y)\right). \]
Take $K > \max(1,L)$.
Take $\varepsilon_0>0$ such that if $d(x,y) < \varepsilon_0$ then $d(f(x),f(y))  < K d(x,y)$.
If $U\subset X$ satisfies $\diam\left(U, d_{[-N,N]^{k}\times \{0\}}\right) < \varepsilon/K^n$ for some $n,N >0$ and
$0<\varepsilon < \varepsilon_0$ then
$\diam\left(U,d_{[-N,N]^{k}\times [0,n]}\right) < \varepsilon$.
Hence for $0<\varepsilon < \varepsilon_0$
\[ \#\left(X, d_{[-N,N]^{k}\times [0,n]},\varepsilon\right) \leq
    \#\left(X,d_{[-N,N]^{k}\times \{0\}}, \varepsilon/K^n\right). \]
We choose positive numbers $\varepsilon_1>\varepsilon_2>\dots\to 0$ satisfying
\[  \lim_{i\to \infty} \frac{S(X,T,d,\varepsilon_i)}{\log(1/\varepsilon_i)}  = \underline{\mdim}(X,T,d). \]
Fix $0<\varepsilon <\varepsilon_0$.
We choose integers $n_i\to \infty$ satisfying
$\varepsilon_i K^{n_i} \leq \varepsilon < \varepsilon_i K^{n_i-1}$.
Then for every $i\geq 1$
\[  \#\left(X,d_{[-N,N]^k\times [0,n_i]},\varepsilon\right) \leq
    \#\left(X, d_{[-N,N]^{k}\times [0,n_i]},\varepsilon_i K^{n_i}\right) \leq
    \#\left(X,d_{[-N,N]^{k}\times \{0\}}, \varepsilon_i \right). \]
From $\varepsilon < \varepsilon_i K^{n_i-1}$,
\[ n_i > \frac{\log(1/\varepsilon_i) + \log\varepsilon + \log K}{\log K}. \]
It follows that
\[ \lim_{N\to \infty} \frac{\log \#\left(X,d_{[-N,N]^k\times [0,n_i]},\varepsilon\right)}{(n_i+1)(2N+1)^k}
   \leq \frac{S(X,T,d,\varepsilon_i)}{\log(1/\varepsilon_i)}   \cdot
    \frac{\log K \cdot \log(1/\varepsilon_i)}{\log(1/\varepsilon_i) + \log\varepsilon + 2\log K}.    \]
Letting $i\to \infty$
\[ S(X,(T,f), d, \varepsilon) \leq \underline{\mdim}(X,T,d) \cdot \log K. \]
Since $0<\varepsilon< \varepsilon_0$ and $K>\max(1, L)$ are arbitrary, this proves the statement.
\end{proof}

When $k=0$, Proposition \ref{proposition: entropy and metric mean dimension}
is just a standard relation between topological entropy and box dimension
(\cite[Theorem 5.6]{Fathi}, \cite[Theorem 3.2.9]{Katok--Hasselblatt}).

\begin{example}
Let $M$ be a compact $C^1$-manifold (of finite dimension).
Consider the $\mathbb{Z}^k$-shift $\sigma$ on $M^{\mathbb{Z}^k}$.
Let $A\subset \mathbb{Z}^k$ be a finite set and $F:M^A \to M$ a $C^1$-map.
We define an endomorphism $f:M^{\mathbb{Z}^k}\to M^{\mathbb{Z}^k}$ of $\sigma$ by the \textit{smooth local rule} $F$:
\[ f\left((x_u)_{u\in \mathbb{Z}^k}\right) = \left(F\left((x_v)_{v\in u+A}\right)\right)_{u\in \mathbb{Z}^k}. \]
Then the topological entropy $h_{\mathrm{top}}(\sigma, f)$ is finite:
Take some distance $d$ on $M$ which comes from a Riemmanian metric and define a distance $D$ on $M^{\mathbb{Z}^k}$ by
\[ D(x,y) = \sum_{u\in \mathbb{Z}^k} 2^{-|u|} d(x_u, y_u). \]
The metric mean dimension $\underline{\mdim}\left(M^{\mathbb{Z}^k}, \sigma, D\right)$
is equal to $\dim M$ (and hence finite) and the map $f$ is Lipschitz with respect to $D$.
\end{example}

\section{Proof of Proposition \ref{proposition: counter-example of multidimensional Mane2}}
\label{section: proof of Proposition counter example}

Here we prove Proposition \ref{proposition: counter-example of multidimensional Mane2}.
Our construction is based on Lindenstrauss--Weiss \cite[Proposition 3.5]{Lindenstrauss--Weiss}.

\subsection{Width dimension} \label{subsection: width dimension}

Let $M= \mathbb{R}^2/\mathbb{Z}^2$ with the standard flat distance $d$.
(The diameter of $M$ is $1/\sqrt{2}$.)
We denote by $\ell^\infty$ the sup-distance on the product space $M^n$:
For $x=(x_0,\dots, x_{n-1}), y= (y_0,\dots,y_{n-1})\in M^n$
\[ \ell^\infty(x,y) = \max_{0\leq i\leq n-1} d(x_i,y_i). \]

\begin{lemma} \label{lemma: widim of M^n}
For $0<\varepsilon < 1/2$
\[   \widim_\varepsilon\left(M^n, \ell^\infty\right) = 2n. \]
\end{lemma}

\begin{proof}
There exists a distance-nondecreasing continuous map from $\left([0,1/2]^{2n}, \text{sup-distance}\right)$
to $\left(M^n,\ell^\infty\right)$.
Hence
\[ \widim_\varepsilon \left(M^n,\ell^\infty\right) \geq \widim_\varepsilon\left([0,1/2]^{2n}, \text{sup-distance}\right). \]
The right-hand side is equal to $2n$ for $\varepsilon< 1/2$ by \cite[Lemma 3.2]{Lindenstrauss--Weiss}.
\end{proof}

\subsection{Proof of Proposition \ref{proposition: counter-example of multidimensional Mane2}}
\label{Proof of Proposition counter-example of multidimensional Mane2}

Let $h:M\to M$ be a hyperbolic toral automorphism.
An important fact for us is that periodic points of $h$ are dense in $M$.
We define $h_n:M^n\to M^n$ by
$h_n(x_0,\dots, x_{n-1}) = (h(x_0),\dots,h(x_{n-1}))$.
Consider the two-sided infinite product $M^\mathbb{Z}$ and let $\sigma:M^\mathbb{Z}\to M^\mathbb{Z}$
be the shift: $\sigma(x)_n= x_{n+1}$.
Define $h_{\mathbb{Z}}:M^\mathbb{Z}\to M^\mathbb{Z}$ by
$h_\mathbb{Z}(x) = (h(x_n))_{n\in \mathbb{Z}}$.
The transformations $\sigma$ and $h_\mathbb{Z}$ generate an expansive $\mathbb{Z}^2$-action on $M^\mathbb{Z}$.
We will construct an appropriate subsystem $X$.
We define a distance $D$ on $M^\mathbb{Z}$ by
\[ D(x,y) = \sum_{n\in \mathbb{Z}} 2^{-|n|} d(x_n,y_n). \]

For $x=(x_n)_{n\in \mathbb{Z}}\in M^\mathbb{Z}$ and integers $a<b$ we denote
\[  x_a^b = (x_a, x_{a+1},\dots, x_b). \]
Let $A\subset M^n$ be a $h_n$-invariant closed subset.
We define a $\mathbb{Z}^2$-invariant closed subset $X(A)\subset M^\mathbb{Z}$
as the set of $x$ satisfying
\[ \exists l\in \mathbb{Z}: \,  \forall m\in \mathbb{Z}: \quad  x_{l+mn}^{l+(m+1)n-1} \in A. \]
In particular if $A= M$ then $X(A) = M^\mathbb{Z}$.

We inductively define positive integers $L_n$ and closed $h_{3^n L_0\cdots L_n}$-invariant sets
$A_n\subset M^{3^n L_0\cdots L_n}$ such that
\begin{equation} \label{eq: periodic points are dense}
    \text{periodic points of $h_{3^n L_0\cdots L_n}$ are dense in $A_n$.}
\end{equation}
First we set $L_0 = 1$ and $A_0=M$.
Suppose we have already defined $L_n$ and $A_n$.
It follows from (\ref{eq: periodic points are dense}) that
there exists a finite set $B_n = \{y^{(1)}, \dots, y^{(b_n)}\}$ in $A_n^3 \subset M^{3^{n+1}L_0\cdots L_n}$ such that
\begin{itemize}
   \item $B_n$ is $(1/n)$-dense in $A_n^3$ with respect to the distance $\ell^\infty$, namely, for every $x\in A_n^3$
   there exists $y\in B_n$ with $\ell^\infty(x,y) < 1/n$.
   \item $B_n$ is $h_{3^{n+1}L_0\cdots L_n}$-invariant. In particular every point $y^{(i)}$ is $h_{3^{n+1}L_0\cdots L_n}$-periodic.
\end{itemize}
We choose $L_{n+1}$ sufficiently larger than $b_n$. (Indeed $L_{n+1} > 2^{n+1}b_n$ is enough. But the detail of the choice is
not important.)
We define closed (but not necessarily invariant) set $C_n\subset (A_n^3)^{L_{n+1}}$ as the set of points
$x = (x_0, \dots, x_{L_{n+1}-1})$, $x_i\in A_n^3$, satisfying
\[ x_{L_{n+1}-i} = y^{(i)} \quad (\forall 1\leq i\leq b_n). \]
We define a closed invariant set $A_{n+1}\subset (A_n^3)^{L_{n+1}}$ by
\[ A_n = \bigcup_{m=0}^\infty h_{3^{n+1}L_0\cdots L_{n+1}}^m\left(C_n\right), \quad
    (\text{this becomes a finite union}). \]
Periodic points of $h_{3^{n+1}L_0\cdots L_{n+1}}$ are dense in $A_{n+1}$.
So we can continue the induction.

The closed $\mathbb{Z}^2$-invariant sets $X(A_n)$ form a decreasing sequence:
\[  M^\mathbb{Z} = X(A_0) \supset X(A_1) \supset X(A_2) \supset \dots. \]
We set $X = \bigcap_{n=0}^\infty X(A_n)$.
This is a closed $\mathbb{Z}^2$-invariant set of $M^\mathbb{Z}$.

\begin{claim}  \label{claim: minimality}
For any $x\in X(A_n)$ and any $y\in X(A_{n+1})$ the exists $p\in \mathbb{Z}$ satisfying
\[  D(x, \sigma^p y) < \frac{3}{n} + 2^{1-3^n L_0\cdots L_n}. \]
Since the right-hand side goes to zero as $n\to \infty$, this shows that $(X,\sigma)$ is minimal .
\end{claim}

\begin{proof}
Let $x=(x_m)_{m\in \mathbb{Z}}\in X(A_n)$ $(x_m\in M)$.
There exists $l\in \mathbb{Z}$ such that
\[ \forall m\in \mathbb{Z}: \> x_{l+m3^n L_0\cdots L_n}^{l+(m+1)3^n L_0\cdots L_n-1} \in A_{n}. \]
We can assume $-2 \cdot 3^n L_0\cdots L_n <  l \leq - 3^n L_0\cdots L_n$.
Since $B_n$ is $(1/n)$-dense in $A_n^3$, we can find $y^{(i)}\in B_n$ which is $(1/n)$-close to $x_l^{l+3^{n+1}L_0\cdots L_n-1}$
with respect to the sup-distance $\ell^\infty$.
From the definition of $A_{n+1}$, any point $y\in X\left(A_{n+1}\right)$ ``contains'' $y^{(i)}$ somewhere, namely
there exists $q\in \mathbb{Z}$ with $y_q^{q+3^{n+1} L_0\cdots L_n-1} = y^{(i)}$.
Then
\begin{equation*}
    \begin{split}
      D(x, \sigma^{q-l} y) &\leq \sum_{m=l}^{l+3^{n+1}L_0\cdots L_n-1} 2^{-|m|} \ell^\infty\left(x_l^{l+3^{n+1}L_0\cdots L_n-1}, y^{(i)}\right)
      + \sum_{m<l \text{ or } m\geq l+3^{n+1}L_0\cdots L_n} 2^{-|m|}  \\
      &<    \frac{3}{n}  + \sum_{|m|> 3^n L_0\cdots L_n} 2^{-|m|} = \frac{3}{n} + 2^{1-3^n L_0\cdots L_n}.
    \end{split}
\end{equation*}
\end{proof}

\begin{claim} \label{claim: positive mean dimension}
The mean dimension $\mdim(X, \sigma\circ h_\mathbb{Z})$ is positive.
\end{claim}

\begin{proof}
For $N\geq 1$ we define a distance $D_N$ on $X$ by
\[ D_N(x,y) = \max_{0\leq m <N} D\left((\sigma\circ h_\mathbb{Z})^m x, (\sigma\circ h_\mathbb{Z})^m y\right). \]
The mean dimension $\mdim(X, \sigma\circ h_\mathbb{Z})$ is given by
\[ \mdim(X, \sigma\circ h_\mathbb{Z}) = \lim_{\varepsilon\to 0}
   \left(\lim_{N\to \infty} \frac{\widim_\varepsilon(X, D_N)}{N}\right). \]

We inductively define $I_n\subset \{0,1,2,\dots, 3^n L_0\cdots L_n-1\}$ by $I_0 = \{0\}$ and
\[ I_{n+1} = \bigcup_{m=0}^{3L_{n+1}-3b_n-1}  \left(m 3^n L_0\cdots L_n + I_n\right). \]
Roughly, $I_n$ is the positions of ``free variables'' of $A_n \subset M^{3^n L_0\cdots L_n}$.
We have $|I_{n+1}| = (3L_{n+1}-3b_n) |I_n|$. Hence for $n\geq 1$
\[ |I_n| = (3L_n-3b_{n-1})(3L_{n-1}-3b_{n-2}) \cdots (3L_1-3b_0). \]

Choose a point $z\in X$ satisfying $z_{m 3^n L_0\cdots L_n}^{(m+1)3^n L_0\cdots L_n -1} \in A_n$ for all $m\in \mathbb{Z}$ and $n\geq 0$.
We define a continuous map $f:M^{I_n}\to X$ by
\[ f(x)_m = \begin{cases}
               h^{-m}(x_m)  &(m\in I_n) \\
               z_m  &(m\not \in I_n).
               \end{cases}    \]
For $x,y\in M^{I_n}$
\[  \ell^\infty(x,y) \leq D_{3^n L_0\cdots L_n}\left(f(x), f(y)\right). \]
Then for $0<\varepsilon < 1/2$
\[ \widim_\varepsilon \left(X, D_{3^n L_0\cdots L_n}\right) \geq \widim_\varepsilon \left(M^{I_n}, \ell^\infty\right) = 2|I_n|
   \quad (\text{Lemma \ref{lemma: widim of M^n}}). \]
So
\[ \lim_{n\to \infty} \frac{\widim_\varepsilon \left(X, D_{3^n L_0\cdots L_n}\right)}{3^n L_0\cdots L_n}
    \geq \lim_{n\to \infty} \frac{2|I_n|}{3^n L_0\cdots L_n} = 2 \prod_{n=1}^\infty \left(1-\frac{b_{n-1}}{L_n}\right). \]
Since we chose $L_n$ sufficiently larger than $b_{n-1}$, the right-hand side is positive.
\end{proof}

\begin{proof}[Proof of Proposition \ref{proposition: counter-example of multidimensional Mane2}]
Set $T=\sigma\circ h_\mathbb{Z}$ and $f = \sigma$.
Then $(X,T)$ is positive mean dimensional (Claim \ref{claim: positive mean dimension}) and
has a jointly expansive and minimal (Claim \ref{claim: minimality}) automorphism $f$.
This proves the proposition.
\end{proof}

\begin{remark}   \label{remark: how to modify the construction}
Here are some remarks on the construction:
\begin{enumerate}
  \item The transformation $f=\sigma$ is also positive mean dimensional on $X$.
  \item We can slightly modify the above construction so that $T= \sigma\circ h_\mathbb{Z}$ also becomes minimal on $X$.
          The modified construction is roughly as follows:
          We choose $L_{n+1}$ sufficiently larger than $b_n^2$ and define $C_n  \subset \left(A_n^3\right)^{L_{n+1}}$
          as the set of points $x = (x_0,\dots, x_{L_{n+1}-i})$, $x_i\in A_n^3$, satisfying
          \[ x^{L_{n+1}-1 -(i-1)b_n}_{L_{n+1}-i b_n} = (\underbrace{y^{(i)}, \dots, y^{(i)}}_{b_n} ) \quad (\forall 1\leq i\leq b_n). \]
          We define $A_{n+1}$ and $X$ as before.
  \item  The action of $h_\mathbb{Z}$ on $X$ is zero mean dimensional.
            So the above $X$ does not provide the proof of Proposition
             \ref{proposition: second counter-example of multidimensional Mane2}.
             We will prove it by modifying the above construction.
             A basic idea is as follows: We consider $X\times X$ with the $\mathbb{Z}^2$-action defined by
             \begin{equation}  \label{eq: joining}
                 (m, n)\cdot (x,y)  = \left(\sigma^m h_{\mathbb{Z}}^n(x), \sigma^{m+n} h_\mathbb{Z}^n(y)\right).
             \end{equation}
             Then we can check that every directional mean dimension of this action is positive.
             But (\ref{eq: joining}) is not minimal (or, at least, we cannot prove its minimality).
             We need to modify the construction so that it becomes minimal.
             In other words the first and second factor of (\ref{eq: joining}) should be disjoint.
             This is the main task of \S \ref{subsection: proof of Proposition second counter-example of Mane2}.
\end{enumerate}
\end{remark}

\section{Directional mean dimension and the proof of Proposition
\ref{proposition: second counter-example of multidimensional Mane2}}
\label{section: directional mean dimension and multidimensional Mane2}

\subsection{Directional mean dimension}
\label{subsection: directional mean dimension}

Here we introduce the notion of \textbf{directional mean dimension}
by mimicking the definition of \textit{directional entropy}
\cite{Milnor, Boyle--Lind}.
We recommend readers to review the definitions in \S \ref{section: mean dimension}.
The concept of directional mean dimensional was suggested to us by Doug Lind.

Let $(X,d)$ be a compact metric space and $T:\mathbb{Z}^2\times X\to X$
a continuous action.
Let $L\subset \mathbb{R}^2$ be a line.
Let $r>1/\sqrt{2}$ and set
\[ B_r(L) = \{x\in \mathbb{R}^2|\, \exists y\in L: |x-y| < r\}. \]
We define the directional mean dimension $\mdim(X,T,L)$ by
\[ \mdim(X,T,L) = \lim_{\varepsilon \to 0}
  \left(\liminf_{N\to \infty} \frac{\widim_\varepsilon\left(X, d_{B_r(L)\cap (-N,N)^2}\right)}{\mathrm{Length}\left(L\cap (-N,N)^2\right)}\right). \]

\begin{remark}  \label{remark: directional mean dimension}
The following properties can be easily checked:

\begin{enumerate}
    \item The value of $\mdim(X,T,L)$ is independent of $r>1/\sqrt{2}$ and the choice of the distance $d$
             (compatible with the underling topology).
    \item  If $L$ and $L'$ are two parallel lines in $\mathbb{R}^2$ then $\mdim(X,T, L) = \mdim(X,T,L')$.
             So it is enough to consider lines passing through the origin.
    \item  If $L = \mathbb{R} u$ for some $u\in \mathbb{Z}^2$ then
             \[ \mdim(X, T, L) = |u| \, \mdim\left(X,T|_{\mathbb{Z} u}\right). \]
             Here $\mdim\left(X, T|_{\mathbb{Z} u}\right)$ is the mean dimension of the restriction of $T$ on the subgroup
             $\mathbb{Z} u\subset \mathbb{Z}^2$.
\end{enumerate}
\end{remark}

\subsection{Proof of Proposition \ref{proposition: second counter-example of multidimensional Mane2}}
\label{subsection: proof of Proposition second counter-example of Mane2}

Here we prove Proposition \ref{proposition: second counter-example of multidimensional Mane2}
by modifying the construction in \S \ref{Proof of Proposition counter-example of multidimensional Mane2}.
The argument is a bit more technical.
We recommend readers to check Remark \ref{remark: how to modify the construction} (3).

We continue to use the notations introduced in \S \ref{Proof of Proposition counter-example of multidimensional Mane2}.
We briefly recall them:
$M=\mathbb{R}^2/\mathbb{Z}^2$ with a hyperbolic toral automorphism $h$.
The two-sided infinite product $M^\mathbb{Z}$ admit the shift $\sigma$ and $h_\mathbb{Z}$ (the component-wise action of $h$),
which generate an expansive $\mathbb{Z}^2$-action.
For a closed and $h_n$-invariant subset $A\subset M^n$ we defined the $\mathbb{Z}^2$-invariant closed subset $X(A)\subset M^\mathbb{Z}$.
The torus $M$ has the standard flat distance $d$ and we defined the distance $D$ on $M^\mathbb{Z}$ by
$D(x,y) = \sum 2^{-|n|}d(x_n,y_n)$.

We will inductively define positive integers $L_n$ and closed $h_{3^n L_0\cdots L_n}$-invariant subsets
$A_n$ and $A'_n$ in $M^{3^n L_0\cdots L_n}$ such that periodic points of $h_{3^n L_0\cdots L_n}$
are dense both in $A_n$ and $A'_n$.
First we set $L_0=1$ and $A_0=A'_0=M$.
Suppose we have already defined $L_n$, $A_n$ and $A'_n$.
There exist $h_{3^{n+1}L_0\cdots L_n}$-invariant subsets $B_n = \{y^{(1)},\dots, y^{(b_n)}\}\subset A_n^3$ and
$B'_n = \{z^{(1)},\dots, z^{(b_n)}\}\subset (A'_n)^3$ such that
$B_n$ and $B'_n$ are $(1/n)$-dense in $A_n^3$ and $(A_n')^3$ respectively with respect to the distance $\ell^\infty$
on $M^{3^{n+1}L_0\cdots L_n}$.
We choose $a_n\geq b_n$ which is a period of all $y^{(i)}$ and $z^{(i)}$
(for simplicity of the notation we set $H = h_{3^{n+1}L_0\cdots L_n}$):
\[  H^{a_n} \left(y^{(i)}\right) = y^{(i)}, \quad
     H^{a_n}\left(z^{(i)}\right) = z^{(i)}. \]
We choose $L_{n+1}$ sufficiently larger than $a_n^2 b_n$.
We define closed subsets $C_n\subset (A_n^3)^{L_{n+1}}$ and $C'_n\subset ((A_n')^3)^{L_{n+1}}$ as follows:
A point $x=(x_0,\dots, x_{L_{n+1}-1})$, $x_i\in A_n^3$,
belongs to $C_n$ if for all $1\leq i\leq b_n$
\[ x_{L_{n+1}-i a_n^2}^{L_{n+1}-1-(i-1)a_n^2} = (\underbrace{y^{(i)}, \dots, y^{(i)}}_{a_n^2}). \]
A point $x=(x_0,\dots, x_{L_{n+1}-1})$, $x_i\in (A_n')^3$, belongs to $C'_n$ if for all $1\leq i\leq b_n$
\[  x_{L_{n+1}-i a_n^2}^{L_{n+1}-1-(i-1)a_n^2} =
   (\underbrace{z^{(i)},\dots, z^{(i)}}_{a_n}, \underbrace{H(z^{(i)}), \dots, H(z^{(i)})}_{a_n}, \dots,
   \underbrace{H^{a_n-1}(z^{(i)}), \dots, H^{a_n-1}(z^{(i)})}_{a_n}). \]
We define $A_{n+1}$ and $A_{n+1}'$ by
\[ A_{n+1} = \bigcup_{m=0}^\infty h_{3^{n+1}L_0\cdots L_{n+1}}^m\left(C_n\right), \quad
   A_{n+1}' = \bigcup_{m=0}^\infty h_{3^{n+1}L_0\cdots L_{n+1}}^m\left(C'_n\right). \]

$X(A_n)$ and $X(A_n')$, $n\geq 0$, form decreasing sequences of closed $\mathbb{Z}^2$-invariant subsets of $M^\mathbb{Z}$.
We set $X = \bigcap_{n=0}^\infty X(A_n)$ and $X' = \bigcap_{n=0}^\infty X(A'_n)$.
We define commuting homeomorphisms $T_1$ and $T_2$ on $M^{\mathbb{Z}}\times M^{\mathbb{Z}}$ by
\[ T_1(x,y) = \left(\sigma(x), \sigma(y)\right), \quad
   T_2(x,y) =\left (h_\mathbb{Z}(x), \sigma\circ h_\mathbb{Z}(y)\right). \]
We will prove that the expansive $\mathbb{Z}^2$-action $(T_1,T_2)$ on $X\times X'$
satisfies the statement of Proposition \ref{proposition: second counter-example of multidimensional Mane2},
namely it is minimal and its every directional mean dimension is positive.
We define a distance $D$ on $M^\mathbb{Z} \times M^\mathbb{Z}$ by
\[ D\left((x,y), (z,w)\right) = \max\left(D(x,z), D(y,w)\right) \quad
    (x,y,z,w \in M^\mathbb{Z}). \]

\begin{claim}
For any $(x,y)\in X(A_n)\times X(A'_n)$ and $(z,w)\in X(A_{n+1})\times X(A_{n+1}')$
there exist integers $p$ and $q$ satisfying
\[ D\left((x,y), T_1^p \circ T_2^q (z,w)\right) < \frac{3}{n} + 2^{1-3^n L_0\cdots L_n}. \]
The right-hand side goes to zero as $n\to \infty$.
Therefore the $\mathbb{Z}^2$-action $\left(X\times X', (T_1,T_2)\right)$ is minimal.
\end{claim}

\begin{proof}
Let $x=(x_m)_{m\in \mathbb{Z}}$, $y=(y_m)_{m\in \mathbb{Z}}$,
$z = (z_m)_{m\in \mathbb{Z}}$ and $w = (w_m)_{m\in \mathbb{Z}}$
with $x_m, y_m, z_m, w_m \in M$.
There exist integers $l_1$ and $l_2$ in $(-2\cdot 3^n L_0\cdots L_n, -3^n L_0\cdots L_n]$ such that
\[   \forall m\in \mathbb{Z}: \>
   x_{l_1+m3^nL_0\cdots L_n}^{l_1+(m+1)3^nL_0\cdots L_n-1}\in A_n, \quad
  y_{l_2+m3^nL_0\cdots L_n}^{l_2+(m+1)3^nL_0\cdots L_n-1}\in A_n'. \]
Since $B_n$ and $B'_n$ are $(1/n)$-dense in $A_n^3$ and $(A_n')^3$ respectively, we can find some points in $B_n$ and $B'_n$
(say, $y^{(i)}$ and $z^{(j)}$) which are $(1/n)$-close to $x_{l_1}^{l_1+3^{n+1}L_0\cdots L_n-1}$
and $y_{l_2}^{l_2+3^{n+1}L_0\cdots L_n-1}$ respectively.

By applying $T_1$ and $T_2$ to $(z,w)$ in an appropriate number of times, we can assume that there exist integers $s_1$ and $s_2$
such that
\[ z_0^{3^{n+1}L_0\cdots L_n a_n^2-1} = ( \underbrace{H^{s_1}(y^{(i)}), \dots, H^{s_1}(y^{(i)})}_{a_n^2}), \]
\begin{equation*}
   \begin{split}
    w_{l_2-l_1}^{l_2-l_1+3^{n+1}L_0\cdots L_n a_n^2-1} = (\underbrace{H^{s_2}(z^{(j)}),\dots, H^{s_2}(z^{(j)})}_{a_n},
    \underbrace{H^{s_2+1}(z^{(j)}), \dots, H^{s_2+1}(z^{(j)})}_{a_n}, \dots, \\
    \underbrace{H^{s_2+a_n-1}(z^{(j)}),\dots, H^{s_2+a_n-1}(z^{(j)})}_{a_n}).
   \end{split}
\end{equation*}
By applying $T_1$ and $T_2$ to $(z,w)$ in an appropriate number of times again , we can assume
\[ z_0^{3^{n+1}L_0\cdots L_n-1} = y^{(i)}, \quad w_{l_2-l_1}^{l_2-l_1+3^{n+1}L_0\cdots L_n-1} = z^{(j)}. \]
Finally, by applying $T_1^{-l_1}$ to $(z,w)$, we can assume
\[ z_{l_1}^{l_1+3^{n+1}L_0\cdots L_n-1} = y^{(i)}, \quad w_{l_2}^{l_2+3^{n+1}L_0\cdots L_n-1} = z^{(j)}. \]
Then
\[ D(x,z) <  \sum_{m=l_1}^{l_1+3^{n+1}L_0\cdots L_n-1} 2^{-|m|} \frac{1}{n} + \sum_{m<l_1 \text{ or } m\geq l_1+3^{n+1}L_0\cdots L_n} 2^{-|m|}, \]
\[  D(y, w) < \sum_{m=l_2}^{l_2+3^{n+1}L_0\cdots L_n-1} 2^{-|m|} \frac{1}{n} + \sum_{m<l_2 \text{ or } m\geq l_2+3^{n+1}L_0\cdots L_n} 2^{-|m|}. \]
Both are bounded by
\[ \frac{3}{n} + \sum_{|m|>3^nL_0\cdots L_n} 2^{-|m|} = \frac{3}{n} + 2^{1-3^n L_0\cdots L_n}. \]
\end{proof}

The proof of Proposition \ref{proposition: second counter-example of multidimensional Mane2} is completed by
the next claim.

\begin{claim}
For any line $L\subset \mathbb{R}^2$ the directional mean dimension
$\mdim\left(X\times X', (T_1,T_2), L\right)$ is positive.
\end{claim}

\begin{proof}
As we remarked in Remark \ref{remark: directional mean dimension} (2),
we can assume that $L$ passes through the origin.
For a subset $\Omega\subset \mathbb{R}^2$ we define a distance $D_\Omega$ on $X\times X'$ by
\begin{equation*}
   \begin{split}
    D_\Omega\left((x,y), (z,w)\right) & = \sup_{(m,n)\in \Omega\cap \mathbb{Z}^2} D\left(T_1^m T_2^n(x,y), T_1^m T_2^n(z,w)\right)  \\
    &= \sup_{(m,n)\in \Omega\cap \mathbb{Z}^2}
       \max\left(D(\sigma^m h_\mathbb{Z}^n(x),\sigma^m h_\mathbb{Z}^n(z)),
               D(\sigma^{m+n}h_\mathbb{Z}^n(y), \sigma^{m+n} h_\mathbb{Z}^n(w)\right).
   \end{split}
\end{equation*}
The directional mean dimension $\mdim\left(X\times X', (T_1,T_2), L\right)$ is given by
\[ \lim_{\varepsilon \to 0}\left(\liminf_{N\to \infty}
    \frac{\widim_\varepsilon \left(X\times X', D_{B_1(L) \cap (-N,N)^2}\right)}{\mathrm{Length}(L\cap (-N,N)^2)}\right). \]
This is proportional to
\begin{equation} \label{eq: modified directional mean dimension}
    \lim_{\varepsilon \to 0}\left(\liminf_{N\to \infty}
    \frac{\widim_\varepsilon \left(X\times X', D_{B_1(L) \cap (-N,N)^2}\right)}{N}\right).
\end{equation}
We will prove that (\ref{eq: modified directional mean dimension}) is positive for any $L$.

We inductively define $I_n\subset \{0,1,2,\dots , 3^n L_0\cdots L_n-1\}$ by $I_0=\{0\}$ and
\[ I_{n+1} = \bigcup_{m=0}^{3L_{n+1}-3a_n^2 b_n-1}\left(m 3^n L_0\cdots L_n + I_n\right). \]
$I_n$ is the positions of the free variables of $A_n$.
It follows that $|I_{n+1}| = (3L_{n+1}-3a_n^2 b_n) |I_n|$ and hence
\[ \frac{|I_n|}{3^n L_0\cdots L_n} =
  \left(1-\frac{a_{n-1}^2 b_{n-1}}{L_n}\right)\left(1-\frac{a_{n-2}^2b_{n-2}}{L_{n-1}}\right)\dots
    \left(1-\frac{a_0^2b_0}{L_1}\right). \]
Since we chose $L_{n+1}$ sufficiently larger than $a_n^2 b_n$, we can assume that for all $n\geq 0$
\begin{equation}  \label{eq: I_n is large}
   \frac{|I_n|}{3^n L_0\cdots L_n} > \frac{1}{2}.
\end{equation}
$\{I_n\}_{n=0}^\infty$ forms an increasing sequence. We define $I\subset \mathbb{Z}$ as the union of all $I_n$.

\begin{subclaim} \label{subclaim: large part of [0,N) is contained in I_n}
For any $t\geq 1$ we have $|[0,t)\cap I| > t/4$.
\end{subclaim}

\begin{proof}
Let $n\geq 0$ be the integer satisfying $3^n L_0\cdots L_n  \leq t < 3^{n+1}L_0\cdots L_{n+1}$.

\noindent
\textbf{Case 1:} $t < 2\cdot 3^n L_0\cdots L_n$.
Then by (\ref{eq: I_n is large})
\[ \frac{|[0,t)\cap I|}{t} \geq \frac{|I_n|}{t} > \frac{|I_n|}{ 2\cdot 3^n L_0\cdots L_n}
   > \frac{1}{4}. \]

\noindent
\textbf{Case 2:} $2 \cdot 3^n L_0\cdots L_n \leq t \leq (3L_{n+1}-3a_n^2 b_n) 3^{n}L_0\cdots L_{n}$.
Take the integer $m$ satisfying
$m 3^n L_0\cdots L_n \leq t < (m+1)3^n L_0\cdots L_n$. 
Then
\[  m \geq \frac{t}{3^n L_0\cdots L_n} -1  \geq \frac{t}{2\cdot 3^n L_0\cdots L_n}.  \]
\[ \frac{|[0,t)\cap I|}{t} \geq \frac{m|I_n|}{t} \geq  \frac{|I_n|}{ 2\cdot 3^n L_0\cdots L_n}
   > \frac{1}{4}. \]

\noindent
\textbf{Case 3:} $(3L_{n+1}-3a_n^2 b_n) 3^{n}L_0\cdots L_{n} < t < 3^{n+1}L_0\cdots L_{n+1}$. Then
\[  \frac{|[0,t)\cap I|}{t} = \frac{|I_{n+1}|}{t} > \frac{|I_{n+1}|}{3^{n+1}L_0\cdots L_{n+1}} > \frac{1}{2}. \]
\end{proof}

Choose points $z\in X$ and $w\in X'$ so that
$z_{m 3^n L_0\cdots L_n}^{(m+1) 3^n L_0\cdots L_n-1} \in A_n$ and
$w_{m 3^n L_0\cdots L_n}^{(m+1) 3^n L_0\cdots L_n-1} \in A_n'$ for all
$m\in \mathbb{Z}$ and $n\geq 0$.

\noindent
\textbf{Case 1.} Suppose $L = \{(t, \alpha t)|\, t\in \mathbb{R}\}$ with some $\alpha$.
We assume $\alpha\geq 0$. (The case $\alpha<0$ is the same.)
We set $\beta = \max(1, \alpha)$.
(We recommend readers to assume that $\alpha \geq 1$ and hence $\beta = \alpha$.
This is a more important case.)
Let $N \geq \beta$ be an integer.
Notice that $\left(m, \lfloor \alpha m\rfloor \right) \in B_1(L) \cap (-N,N)^2$ for $m \in  [0, N/\beta)$.
By Subclaim \ref{subclaim: large part of [0,N) is contained in I_n}, $|I\cap [0,N/\beta)| > N/(4\beta)$.
We define a continuous map $f:M^{I\cap [0,N/\beta)} \to X\times X'$ by
\[  f(x)_m = \begin{cases}
                   (h^{-\lfloor \alpha m\rfloor}(x_m), w_m)  & (m\in I\cap [0,N/\beta)) \\
                   (z_m, w_m)                                       & (m\not \in I\cap [0,N/\beta)).
                 \end{cases}
\]
Then for any $x,y\in M^{I \cap [0,N/\beta)}$
\[  \ell^\infty(x,y) \leq D_{B_1(L) \cap (-N,N)^2} \left(f(x), f(y)\right).  \]
It follows from Lemma \ref{lemma: widim of M^n} that for $0<\varepsilon < 1/2$
\[ \widim_\varepsilon \left(X\times X', D_{B_1(L) \cap (-N,N)^2} \right)
    \geq \widim_\varepsilon \left(M^{I\cap [0,N/\beta)}, \ell^\infty\right) = 2 |I\cap [0,N/\beta)| > \frac{N}{2\beta}. \]
This shows that (\ref{eq: modified directional mean dimension}) is larger than or equal to $1/(2\beta)$.

\noindent
\textbf{Case 2.} Suppose $L = \{(0, t)|\, t\in \mathbb{R}\}$.
This case is essentially the same with Claim \ref{claim: positive mean dimension} because of the form
\[ T_2^m(*, x) = \left(*, \sigma^m h_\mathbb{Z}^m (x)\right). \]
But we provide the proof for the completeness. Let $N$ be a natural number.
We define a continuous map $f:M^{I\cap [0,N)} \to X\times X'$ by
\[  f(x)_m = \begin{cases}
                   (z_m, h^{-m}(x_m))  & (m\in I\cap [0,N)) \\
                   (z_m, w_m)           & (m\not \in I\cap [0,N)).
                 \end{cases}
\]
Then for any $x,y\in M^{I \cap [0,N)}$
\[  \ell^\infty(x,y) \leq D_{B_1(L) \cap (-N,N)^2} \left(f(x), f(y)\right).  \]
From Lemma \ref{lemma: widim of M^n}, for $0<\varepsilon < 1/2$
\[ \widim_\varepsilon \left(X\times X', D_{B_1(L) \cap (-N,N)^2} \right)
    \geq \widim_\varepsilon \left(M^{I\cap [0,N)}, \ell^\infty\right) = 2 |I\cap [0,N)| > \frac{N}{2}. \]
Hence (\ref{eq: modified directional mean dimension}) is larger than or equal to $1/2$.
\end{proof}

\vspace{0.5cm}

\address{Tom Meyerovitch \endgraf
Department of Mathematics, Ben-Gurion University of the Negev, Israel}

\textit{E-mail address}: \texttt{mtom@math.bgu.ac.il}

\vspace{0.5cm}

\address{ Masaki Tsukamoto \endgraf
Department of Mathematics, Kyoto University, Kyoto 606-8502, Japan}

\textit{E-mail address}: \texttt{tukamoto@math.kyoto-u.ac.jp}

\end{document}